\documentclass[11pt]{amsart}
\usepackage{color}
\usepackage{geometry}
\usepackage{comment}
\usepackage{graphicx}
\usepackage{epstopdf}

\usepackage{color}
\usepackage{textcomp}
\usepackage{amsthm}
\usepackage{amstext}
\usepackage{amssymb}

\newcommand{\ZZ}{\mathbb{Z}}
\newtheorem{defn}{Definition}
\newcommand{\lm}{\lambda}
\newcommand{\Lm}{\Lambda}
\DeclareMathOperator{\sdim}{sdim}

\DeclareMathOperator{\df}{def}
\DeclareMathOperator{\ch}{ch}
\DeclareMathOperator{\sch}{sch}
\newcommand{\h}{\mathfrak{h}}
\newcommand{\gl}{\mathfrak{gl}}
\newcommand{\g}{\mathfrak{g}}
\newcommand{\sln}{\mathfrak{sl}}
\newcommand{\fw}{\mathcal{F}_W}
\newcommand{\fwo}{\mathcal{F}_{W_1}}
\newcommand{\fwl}{\mathcal{F}_{W_{\Lambda}}}
\newcommand{\eentry}{>}
\newcommand{\dentry}{<}
\newcommand{\str}{\mathfrak{str}}
\newcommand{\tr}{\mathfrak{tr}}
\newcommand{\both}{\times}
\newcommand{\neither}{\circ}
\numberwithin{equation}{section}
\numberwithin{figure}{section}
\newtheorem{lem}{Lemma}

\theoremstyle{plain}
\newtheorem{thm}{Theorem}
\theoremstyle{remark}
\newtheorem{rem}{Remark}
  \theoremstyle{plain}
  \newtheorem*{thm*}{\protect\theoremname}

\newcommand{\dsp}{\displaystyle}

\title{A Superdimension Formula for $\gl(m|n)$-Modules}
\author{Michael Chmutov}
\address{
Department of Mathematics,
University of Minnesota, 
127 Vincent Hall, 
206 Church Street SE,
Minneapolis, MN 55455
USA}
\email{mchmutov@umn.edu}

\author{Rachel Karpman}
\address{
Department of Mathematics,
University of Michigan, Ann Arbor,
530 Church St.,
Ann Arbor, MI 48109-1043
USA}
\email{rkarpman@umich.edu}

\author{Shifra Reif}
\address{
Department of Mathematics,
University of Michigan, Ann Arbor,
530 Church St.,
Ann Arbor, MI 48109-1043
USA}
\email{shifi@umich.edu}

\begin{document}

\begin{abstract}
We give a formula for the superdimension of a finite-dimensional simple $\mathfrak{gl}\left(m|n\right)$-module
using the Su-Zhang character formula. This formula coincides with the superdimension formulas proven by Weissauer and Heidersdorf-Weissauer. As a corollary, we obtain a
simple algebraic proof of a conjecture of Kac-Wakimoto for $\mathfrak{gl}\left(m|n\right)$,
namely, a simple module has nonzero superdimension if and only if
it has maximal degree of atypicality. This conjecture was proven originally
by Serganova using the Duflo-Serganova associated variety.
\end{abstract}
\thanks{The work for this paper was partially supported by RTG grants NSF/DMS-1148634 and NSF/DMS-0943832.}
\maketitle{}

\section{Introduction}
The superdimension of a $\ZZ_2$-graded vector space $V=V_{\bar 0}\oplus V_{\bar 1}$ is defined to be
$$\sdim V:=\dim V_{\bar 0}-\dim V_{\bar 1}.$$ 
It was conjectured in 1994 by Kac and Wakimoto that the superdimension
of a finite-dimensional simple module of a basic Lie superalgebra
$\mathfrak{g}$ is nonzero if and only if the atypicality is maximal \cite{KW}.
This conjecture was supported by a theorem stating that the evaluation
of the so-called Bernstein-Leites (super)character is nonzero exactly
under this condition. At the time, there was no character formula for Lie
superalgebras, and it was not known precisely how the Bernstein-Leites character
was related to the actual character of the module. The conjecture was finally proved by Serganova but without giving a formula for the superdimension \cite{S3}.

Major progress has been made since then on the character theory for basic Lie superalgebras. 
In 1996, Serganova  gave a general character
formula for finite dimensional irreducible representations of $\mathfrak{gl}\left(m|n\right)$ in terms of generalized Kazhdan-Lusztig polynomials \cite{S1,S2}. 
Brundan, in 2003, gave an explicit algorithm for computing
these generalized Kazhdan-Lusztig polynomials 
\cite{B}. In 2007, Su and Zhang used Brundan's algorithm to prove
a character formula which consists of a finite alternating sum of
Bernstein-Leites characters (see Theorem \ref{SZ formula}). 

In this paper we use the Su-Zhang character formula to give a formula
for the superdimension of a simple finite dimensional $\mathfrak{gl} \left(m | n\right)$-module $L\left(\Lambda\right)$
with highest weight $\Lambda$. When $\Lambda$ is of maximal atypicality, the formula consists of a product of two positive terms. The first term, denoted by $s_\Lm$, is the maximal number of monomials that can appear in a Kazhdan-Lusztig polynomial $K_{\Lambda,\mu}$ for any weight $\mu$. As shown by Su and Zhang, this number can be computed easily using Brundan's algorithm (see Equation (\ref{size of s_Lambda})). The second term is equal to the dimension of a simple module of a Lie algebra isomorphic to $\mathfrak{gl}(|m-n|)$. Using the dimension formula for simple Lie algebras, we obtain:

\begin{thm} \label{main thm} Let $\Gamma_\Lambda$  be a maximal set of isotropic roots which are mutually orthogonal and orthogonal to $\Lm+\rho$ and let $M_{\Lm}^+$ be the set of even positive roots of $\g$ orthogonal to $\Gamma_{\Lm}$.  
Then 
\[
|\sdim L(\Lm)| =
\begin{cases} s_{\Lm} \prod_{\alpha \in M_{\Lm}^+}  \frac{ \langle \Lm +\rho, \alpha^{\vee}\rangle}{\langle \Lm+\rho-\rho_{\Lm}^{0}, \alpha^{\vee}\rangle}& \Lm  \emph{ of maximal atypicality}\\
0 & \emph{otherwise}
\end{cases}
\]
where $\rho_\Lambda^0=\frac{1}{2}\sum_{\alpha\in M_\Lambda^{+}} \alpha$.
\end{thm}

This formula coincides with the superdimenstion formula proven in \cite{HW} and \cite{W} using categorical methods and Duflo-Serganova functors, see Remark \ref{translationRemark}. It would be interesting to find a simple proof for Theorem \ref{main thm} which extends to representations of other types of Lie superalgebras.

\subsection*{Acknowledgement.} 
We are grateful to Gal Binyamini for helpful discussions and to Thorsten Heidersdorf for pointing out previous works about the superdimension formula.

\section{Preliminaries}

Let $\g$ denote the Lie superalgebra $\gl(m | n)$, and without loss of generality assume $m \geq n$.  Let $\h$ denote the Cartan subalgebra of $\g$.  We use the standard notation for the odd and even roots, namely
\[\Delta_{\bar{0}} = \{\epsilon_i-\epsilon_j \mid 1 \leq i \neq j \leq m\} \cup \{\delta_i-\delta_j \mid 1 \leq i \neq j \leq n\}\]
\[\Delta_{\bar{1}} = \{\epsilon_i - \delta_j, \delta_j-\epsilon_i \mid 1 \leq i \leq m, 1 \leq j \leq n\}.\]
We normalize the bilinear form on $\h^*$ so that for all $i,j$ we have $(\epsilon_i,\epsilon_j)=\delta_{i,j}$, $(\delta_i,\delta_j)=-\delta_{i,j}$, and $(\epsilon_i,\delta_j)=0$ where $\delta_{i,j}$ is the Kronecker delta function.

We fix our choice of simple roots to be the standard one, that is
\[\{\epsilon_1 - \epsilon_2,\ldots,\epsilon_{m-1}-\epsilon_m,\epsilon_m-\delta_1,\delta_1-\delta_2,\ldots,\delta_{n-1}-\delta_n\}.\]
Let $\Delta_{\bar{0}}^+$ and $\Delta_{\bar{1}}^+$ be the corresponding sets of even and odd positive roots, respectively.  
Let $Q^+$ denote the $\mathbb{N}$-span of the positive roots of $\g$.  Let $p$ denote the parity function on the roots of $\g$, and extend $p$ to $Q^+$ in the natural way. We shall use the standard partial order on $\mathfrak{h}^*$ defined by $\lm \geq \mu$ if $\lm-\mu \in Q^+$.  
 Let $\rho = \frac{1}{2} \sum_{\alpha \in \Delta_{\bar{0}}^+} \alpha - \frac{1}{2} \sum_{\alpha \in \Delta_{\bar{1}}^+} \alpha.$
For $\mu \in \h^*$, we say that $\mu$ is \emph{dominant} (resp. \emph{strictly dominant}) if $\frac{2(\mu,\alpha)}{(\alpha,\alpha)}\geq 0$ (resp. $\frac{2(\mu,\alpha)}{(\alpha,\alpha)}> 0$) for all $\alpha \in \Delta_{\bar{0}}^+$.
Let $R$ and $\check{R}$ be the \textit{Weyl denominator} and \textit{superdenominator}, respectively, that is
\[R = \frac{\prod_{\alpha \in \Delta_{\bar{0}}^+} (1-e^{-\alpha})}{\prod_{\alpha \in \Delta_{\bar{1}}^+}(1+e^{-\alpha})}
\text{ and }
\check{R}=\frac{\prod_{\alpha \in \Delta_{\bar{0}}^+}(1-e^{-\alpha})}{\prod_{\alpha \in \Delta_{\bar{1}}^+} (1-e^{-\alpha})}.\]

A root of a Lie superalgebra is \textit{isotropic} if it is orthogonal to itself.  The isotropic roots of $\gl(m|n)$ are precisely the odd roots.  For a Lie superalgebra $\mathfrak{a}$, we define the \textit{defect} of $\mathfrak{a}$, denoted $\df \mathfrak{a}$, to be the size of a maximal set of isotropic positive roots which are mutually orthogonal.  For $\g = \gl(m | n)$, with $m \geq n$, we have $\df \g = n$.  

Let $L(\Lm)$ be a simple finite-dimensional representation of highest weight $\Lm$. Note that $\Lm$  is a dominant weight, and since we chose the standard set of simple roots, $\Lm+\rho$ is strictly dominant.  Let $\Gamma_\Lm$ be a maximal set of isotropic roots, which are orthogonal to each other and to $\Lm+\rho$. Since $\Lm+\rho$ is strictly dominant, this set is unique.  The  \textit{atypicality} of $\Lm$ is defined to be $r = |\Gamma_{\Lm}|.$

We set $M_{\Lm}$ to be the set of even roots of $\g$ orthogonal to $\Gamma_{\Lm}$, and let $\g_{\Lm}$ be the Lie algebra with root system $M_{\Lm}$.  Note that if $r=n$, $\g_{\Lm} \cong \gl(m-n)$.  Denote $ M_{\Lm}^+:= M_{\Lm} \cap \Delta_{\bar{0}}^+ $, $\rho_{\Lm}^0=\frac{1}{2} \sum_{\alpha \in M_{\Lm}^+} \alpha$ and
$  R_\Lm := \prod_{\alpha\in M_{\Lm}^+}  \left( 1-e^{-\alpha} \right).$
We denote the simple $\g_\Lm$ module of highest weight $\mu $ by $L_\Lm (\mu)$. We use the same notation for a weight $\lambda\in\h$ and its restriction to $\g_\Lm\cap\h$.

Given a weight space decomposition
$L (\Lm) = \bigoplus_{\mu \in Q^{+}} L_{\Lm - \mu},$
the \textit{character} and \textit{supercharacter} of $L(\Lm)$ are  given by
\[\ch L(\Lm) = \sum_{\mu \in Q^+} (\dim L_{\Lm-\mu}) e^{\Lm-\mu},\quad
\sch L (\Lm)= \sum_{\mu \in Q^+} (-1)^{p(\mu)} (\dim L_{\Lm-\mu}) e^{\Lm-\mu}.\]
Note that these characters yield functions on $\mathfrak{h}$, defined by $e^{\lm}(h) = e^{\lm(h)}$ for $h\in\h$ and $\lm\in\h^*$. For $f$ a function on $\h^*$, let $\left.f\right|_0$ denote evaluation at $0$.  We have
\[\left.\ch L(\Lm)\right|_0 = \dim L(\Lm) \text{ and }\left.\sch L(\Lm)\right|_0 = \sdim L(\Lm)\]

The Weyl group $W$ acts on the space $\mathcal{E}$ of rational functions in $e^\lm$, $\lm\in\h$, by $w e^\lm=e^{w\lm}$. Let $\ell$ denote the length function of $W$.  For $W'\subseteq W$ and  $X\in\mathcal{E}$, we denote
\[    \mathcal{F}_{W'} \left(X\right)=\sum_{w\in W'} (-1)^{\ell(w)} w\left(X\right).  \]

\section{Proof of Theorem \ref{main thm}}

The Su-Zhang character formula gives the character of a finite-dimensional irreducible $\g$-module $L(\Lm)$.  We use this formula to derive a formula for the supercharacter, which we evaluate at zero to find the superdimension of $L(\Lm)$.

\subsection{The Su-Zhang Character Formula}

To state the Su-Zhang formula, we need additional notation.  In particular, we shall introduce two subsets of the Weyl group $W$ of $\g$, denoted $S_{\Lm}$ and $C_r$.

We denote the elements of $\Gamma_{\Lm}$ by $\{\beta_1,\ldots,\beta_r\}$ where $\beta_k = \epsilon_{i_k} - \delta_{j_k}$, and $ j_1 < j_2 < \cdots < j_r$. Note that this notation imposes an order on $\Gamma_{\Lm}$.  We embed $Sym_r$ in the Weyl group $W$ of $\g$ by sending the transposition $(k,\ell) \in Sym_r$ to the product of reflections
 $s_{\epsilon_{i_k}-\epsilon_{i_{\ell}}}s_{\delta_{j_k}-\delta_{j_{\ell}}}$.  
 Thus $(k,\ell)$ maps to an element of $W$ which interchanges $\beta_k$ and $\beta_{\ell}$.  Note that for $\mu\in\Lm-\ZZ\Gamma_\Lm$  and $\sigma \in Sym_r$, we have $\sigma(\mu+\rho) \in \Lm+\rho - \mathbb{Z}\Gamma_{\Lm}$.

To define $S_{\Lm}$, we recall the \textit{weight diagram} construction introduced in \cite{BS}.  We write 
$$\displaystyle \Lm+\rho = \sum_{i=1}^m a_i\epsilon_i - \sum_{j=1}^n b_j\delta_j. $$  
To construct the weight diagram of $\Lm+\rho$, we assign a symbol to each integer $k$, according to the rule

\begin{eqnarray*}
\eentry \quad & \text{if }k \in \{a_1,\ldots,a_m\},\, k \not \in \{b_1,\ldots,b_n\}\\
\dentry \quad & \text{if }k \not\in \{a_1,\ldots,a_m\},\, k \in \{b_1,\ldots,b_m\}\\
\both \quad &\text{if }k \in \{a_1,\ldots,a_m\},\,  k \in  \{b_1,\ldots,b_n\}\\
\neither \quad & \text{if }k \not\in \{a_1,\ldots,a_m\},\, k \not\in \{b_1,\ldots,b_n\}
\end{eqnarray*}
We number the $\both$'s from left to right.  For example, if $$\Lm+\rho = 6 \epsilon_1+5\epsilon_2+3\epsilon_3+2\epsilon_4+\epsilon_5-\delta_1-3\delta_2-6\delta_3-8\delta_4,$$ then the corresponding weight diagram is 

\[\begin{array}{cccccccccc}
\ldots & \both_1 & \eentry & \both_2& \neither &\eentry & \both_3 & \neither & \dentry & \ldots\\
 & \scriptstyle{1} & \scriptstyle{2} & \scriptstyle{3} & \scriptstyle{4} & \scriptstyle{5} & \scriptstyle{6} & \scriptstyle{7} & \scriptstyle{8} & 
\end{array}\]
For $k \leq \ell$, we say that $\beta_k$ and $\beta_{\ell}$ are \textit{strongly connected} if for every $k < i \leq \ell$, the number of entries $\circ$ between $\both_k$ and $\both_i$ in the weight diagram for $\Lm+\rho$ is less than or equal to the number of entries $\both$ between $\both_k$ and $\both_i$.  In the example above, $\beta_1$ and $\beta_2$ are strongly connected, as are $\beta_1$ and $\beta_3$; however, $\beta_2$ and $\beta_3$ are not strongly connected.  For each $1\leqslant s\leqslant r$, let $\max_s^\Lm$ be the largest $t\leq r$ such that $\beta_s$ and $\beta_t$ are strongly connected. 

\begin{defn}
Let
\[S^{\Lm} :=\{\sigma\in Sym_r\mid \sigma^{-1}(s) < \sigma^{-1}(t) \text{ if $s < t$ and $\beta_s$ and $\beta_t$ are strongly connected}\}.\]
By \cite[3.18]{SZ1}, we obtain 
\begin{equation}
\label{size of s_Lambda}
s_{\Lm} := |S_{\Lm}| = \frac{r!}{{\displaystyle \prod_{s=1}^r} (\max_s^{\Lm} - s + 1)}.
\end{equation}
\end{defn}

\begin{rem} \label{translationRemark}
In \cite{HW}, the formula for the superdimension is described using a cup diagram which connects $\times$'s and  $\circ$'s (which are denoted there by $\vee$ and $\wedge$, respectively). In that notation, two roots are strongly connected if and only if the corresponding cups are nested. Thus, $s_\Lambda$ coincides with the formula for $m(\lambda)$ in Section 14 of  \cite{HW}.
\end{rem}

\begin{defn}
Let $C_r$ be the set of cyclic permutations of order $r$, that is, all permutations of the form
\[\pi = (1,\ldots,i_1)(i_1+1,\ldots,i_1+i_2)\cdots(i_1+\ldots+i_{t-1}+1,i_1+\ldots+i_{t-1}+2,\ldots,r)\]
where $i_1,\ldots,i_t \in \mathbb{N}$ and $i_1+\ldots+i_t=r$. For $\pi \in C_r$, we define
\[{r \choose \pi} = \frac{r!}{i_1!i_2!\cdots i_t!}.\]
\end{defn}

We define the operation $\Uparrow$ for $\lm \in \Lm+\rho-\mathbb{N} \Gamma_{\Lm}$ by setting $\mu = \lm_{\Uparrow}$ to be the maximal weight in $\Lm+\rho-\mathbb{N}\Gamma_{\Lm}$ where the coefficients of $\delta_{j_k}$ in $\mu$ are weakly increasing or, equivalently, the coefficients of $\epsilon_{j_k}$ in $\mu$ are weakly decreasing (with the notation of \cite{SZ1}, $\lambda_\uparrow=\left(\lambda+\rho\right)_\Uparrow-\rho$).

For $\lm,\mu \in \Lm+\rho - \mathbb{Z}\Gamma_{\Lm}$, let $\displaystyle \lm-\mu = \sum_{i=1}^r c_i \beta_i$.  We define $|\lm-\mu|$ by $\displaystyle |\lm-\mu| := \sum_{i=1}^r |c_i|$.  

With the notation above, we have:
\begin{thm} \cite[Theorem 4.9]{SZ1} \label{SZ formula} The character of a finite dimensional simple $\g$-module with highest weight $\Lm$ is given by
\[\ch L(\Lm) = \sum_{\sigma \in S^{\Lm}, \pi \in C_r} \frac{1}{r!} {r \choose \pi} (-1)^{|\Lm+\rho-(\pi (\sigma (\Lm+\rho))_{\Uparrow})_{\Uparrow}|+\ell(\pi)}
e^{-\rho}R^{-1}  \cdot \fw \left(\frac{e^{(\pi (\sigma (\Lm+\rho))_{\Uparrow})_{\Uparrow}}}{ \prod_{\beta \in \Gamma_{\Lambda}} (1 + e^{-\beta})}\right)\]
\end{thm}

We transform this character formula to a formula for the supercharacter. For $\nu\in\h^*$ let
\[\chi(\nu):=e^{-\rho}\check{R}^{-1} \cdot \mathcal{F}_W \left( \frac{e^{\nu}}{\prod_{\beta \in \Gamma_{\Lm} }(1-e^{-\beta})}\right).\]

Expanding each term $\frac{1}{1-e^{-\beta}}$ as a geometric series, and changing signs as appropriate, we obtain
\begin{equation}
\label{sch} 
\sch L(\Lambda) = \sum_{\sigma \in S^{\Lambda},\pi \in C_r} \frac{1}{r!}{r \choose \pi} (-1)^{\ell(\pi)} \chi((\pi (\sigma(\Lambda+\rho))_{\Uparrow})_{\Uparrow} ).\end{equation}

\subsection{Evaluation}

We now compute the superdimension of $L(\Lm)$ by evaluating the formula for the supercharacter in (\ref{sch}).  We first show that many of the terms evaluate to the same number.
\begin{lem} 
\label{same} 
For any $\mu \in \Lm - \mathbb{Z}\Gamma_{\Lm}$ with $\mu+\rho$ $M_{\Lm}$-dominant, $\left. \chi(\Lm+\rho) \right|_{0} = \left. \chi(\mu+\rho) \right|_{0}$.
\end{lem}

\begin{proof}
Let $W_{\Lm}$ be the subgroup of $W$ generated by roots from $M_{\Lm}$ and let $W_1$ be a set of left coset representatives, so that $W = W_1W_{\Lm}$. We have
\[
\chi(\mu+\rho) 
= e^{-\rho}\check{R}^{-1} \cdot \fwo \left(\frac{\fwl(e^{\mu+\rho})}{\prod_{\beta \in \Gamma_{\Lm}}(1-e^{-\beta})} \right). 
\]
Since $\Lm$ is dominant, $\mu+\rho-\rho_{\Lm}^0$ is $M_{\Lm}$-dominant, the Weyl character formula implies
\[\fwl(e^{\mu+\rho}) = e^{\rho_{\Lm}^{0}}R_{\Lm} \cdot \ch L_{\Lm}(\mu + \rho - \rho_{\Lm}^{0}).\]
Since $ e^{-\rho}\check{R}^{-1} $ is $W_1$-anti-invariant, we have
\[ \chi(\mu+\rho) =  
 \sum_{w\in W_1} w\left(\frac{e^{-\rho}\check{R}^{-1} \cdot e^{\rho_{\Lm}^{0}}R_{\Lm}}{\prod_{\beta \in \Gamma_{\Lm}}(1-e^{-\beta})} 
 \cdot \ch L_{\Lm}(\mu + \rho - \rho_{\Lm}^{0}) \right).\]
 
The number of zeros minus the number of poles of the term ${e^{-\rho}\check{R}^{-1} \cdot e^{\rho_{\Lm}^{0}}R_{\Lm}} \cdot
{\prod_{\beta \in \Gamma_{\Lm}}(1-e^{-\beta})^{-1}} $ at $0$ is $(m-r)(n-r)$. 
Indeed $\left| \Delta _{\bar 0} ^{+}\right|= \frac{n(n-1)+m(m-1)}{2} $,  $\left| \Delta _{\bar 1} ^{+} \right|=mn$, $ \left| M_\Lm^{+}\right|=\frac{(m-r)(m-r-1)+(n-r)(n-r-1)}{2}$ and  $\left| \Gamma_\Lambda \right| =r$. 
Since $(m-r)(n-r)\ge 0$, we can evaluate $\chi \left(\mu+\rho\right)$ term by term, that is 
\[\left. \chi(\mu+\rho) \right|_0 =  
 \sum_{w\in W_1} w\left(\left. \frac{e^{-\rho}\check{R}^{-1} \cdot e^{\rho_{\Lm}^{0}}R_{\Lm}}{\prod_{\beta \in \Gamma_{\Lm}}(1-e^{-\beta})} \right|_0
 \cdot \left. \ch L_{\Lm}(\mu + \rho - \rho_{\Lm}^{0}) \right|_0 \right).  \]
Since $\Lm-\mu$ is orthogonal to $M_{\Lm}$, we get that
\[\ch L_{\Lm}(\Lm +\rho -\rho_{\Lm}^{0})=e^{\Lm-\mu}\ch L_{\Lm}(\mu +\rho - \rho_{\Lm}^{0}).\]
Since $\left.e^{\Lm-\mu}\right|_0=1$, the evaluation is the same as desired.
\end{proof}

We use the following theorem of Kac and Wakimoto to compute $\left. \chi(\Lm+\rho) \right|_{0}$ (note with the notation of \cite{KW}, 
$\chi(\Lm+\rho)=j_{\Lm}\sch_{\Lm}$). 

\begin{thm}\cite[Theorem 3.3]{KW} 
\label{KWbig} One has
\[\left|\left.\chi(\Lm+\rho)\right|_{0}\right|= \begin{cases}
n!\dim L_{\Lm}(\Lm+\rho-\rho_{\Lm}^{0}) & r = \df \g\\
0 & \emph{ otherwise.}
\end{cases}\]
\end{thm}

Since $(\pi (\sigma (\Lambda+\rho))_{\Uparrow})_{\Uparrow}$ is contained in $\Lm+\rho - \ZZ \Gamma_{\Lm}$ for every $\pi \in C_r$ and $\sigma \in S_{\Lm}$, Lemma $1$ implies that each term $\left.\chi((\pi (\sigma (\Lambda+\rho))_{\Uparrow})_{\Uparrow})\right|_0$ is equal to the constant $\left.\chi(\Lm+\rho)\right|_0$.  Hence, if $r \neq \df \g$, the formula evaluates to $0$, completing the proof for this case.  If $r = \df \g = n$, we have
\[\sdim L(\Lm)= \pm \sum_{\sigma \in S^{\Lambda},\pi \in C_r} \frac{1}{r!}{r \choose \pi} (-1)^{\ell(\pi)} n!\dim L_{\Lm}(\Lm+\rho-\rho_{\Lm}^{0}).\]
By the dimension formula for simple Lie algebras we have
$$ \dim L_{\Lm}(\Lm+\rho-\rho_{\Lm}^{0}) =  \prod_{\alpha \in M_{\Lm}^+}  \frac{ \langle \Lm +\rho, \alpha^{\vee}\rangle}{\langle \Lm+\rho-\rho_{\Lm}^{0}, \alpha^{\vee}\rangle}.$$

To complete the proof of the Theorem \ref{main thm}, it remains to prove the following lemma.
\begin{lem} For $r>0$, we have $\displaystyle \sum_{\pi \in C_r} {r\choose \pi}(-1)^{\ell(\pi)}=1.$
\end{lem}

\begin{proof}
The parity of a permutation $\pi\in Sym_r$ is $r$ plus the number $t$ of cycles of $\pi$. Splitting the sum on the left hand side based on the number of 
cycles we perform the following calculation using generating functions, where $[x^r]$ is the operator that takes the coefficient of $x^r$ of a power series.
\begin{align*}
\sum_{\pi \in C_r} {r\choose \pi}(-1)^{\ell(\pi)} & = \dsp \sum_{t\geqslant 1}\sum_{\substack{r_1+\dots+r_t = r\\ r_i\neq 0}} \frac{r!}{r_1!\dots r_t!} (-1)^{r+t}\\
&= (-1)^r r! \sum_{t\geqslant 1} \left[x^r\right] (1-e^x)^t\\
& = (-1)^r r! \left[x^r\right]\left(\sum_{t\geqslant 1}  (1-e^x)^t\right)\\
& = \dsp (-1)^r r! \left[x^r\right]\left(\frac{(1-e^x)}{1-(1-e^x)}\right)\\
& = (-1)^r r! \left[x^r\right]\left(e^{-x}-1\right)\\
&= 1
\end{align*} 
\end{proof}

\section{Examples}

Let us illustrate our formula with a few examples.  
We use the vector notation for the weights of $\g$, namely
\[(a_1,\cdots,a_m \mid b_1,\ldots,b_n)  :=  \sum_{i=1}^m a_i \epsilon_i - \sum_{i=1}^n b_i \delta_i.\]
 Note that $a_i=b_j$ means that $\epsilon_i-\delta_j \in \Gamma_\Lm$. Shifting the highest weight of a $\g$-module by $\str :=(1,\ldots,1|1,\ldots,1)$ does not change the superdimension of the module. Similarly, shifting the highest weight of a $\g_\Lm$-module by $\tr:=\sum_{i\in M}\epsilon_i$, $ M=\{i\mid \epsilon_i-\delta_j\notin\Gamma(\Lm)\ \forall j\} $ does not change the dimension. Thus, the computations below are done up to a multiple of $\str$ and $\tr$.
 
\subsection{The trivial representation}

The highest weight of  the trivial representation is $\Lm=0$, so $$\Lm +\rho=(m,m-1,\ldots,2,1 \mid 1,2,\ldots,n)$$ and $s_{\Lm} = 1$.
We have that $M_{\Lm} = \left\{ \epsilon_i - \epsilon_j    \mid 1 < i,j \leq m-n \right\}$ and
$ \rho-\rho_{\Lm}^0  =0$.
Thus, our formula gives that the superdimension of the trivial representation is equal to the dimension of the trivial representation of $\mathfrak{g}_\Lm$  which is $1$ as desired.

\subsection{The natural representation} Let $V$ be the natural representation of $\g$.  Let us show that our formula gives $|\sdim(V)|=m-n$.  

The highest weight of $V$ is $\Lm = \epsilon_1$, so we have $\Lm + \rho = (m+1,m-1,\ldots,1|1,2,\ldots,n)$, and $s_{\Lm}=1$.
Then the atypicality is $n$ for $m > n$, and $n-1$ for $m = n$.  In the latter case, we have $r \neq n$, and so our formula gives $0$ for the superdimension, as desired.

For $m > n$,  $M_{\Lm}=\{\epsilon_i-\epsilon_j \mid i,j \leq m-n\},$ and we have
$\Lm+\rho-\rho_{\Lm}^{0} =\epsilon_1$. However, this is the highest weight of the natural representation of $M_{\Lm} \cong \gl(m-n)$, and we get $|\sdim V|=\dim L_{Lm}(\epsilon_1)=m-n$.

\subsection{The adjoint representation} 

Let $V$ be the irreducible component of the adjoint representation of $\g$ corresponding to $\sln(m | n)$ for $m > n$, and to $\mathfrak{psl}(n|n)$ for $m = n$.  Then 
\[|\sdim(V)| = m^2 + n^2 - 2mn - 1 - \delta_{mn}.\] 
The highest weight of $V$ is the highest root $\Lm = \epsilon_1 - \delta_n$
and $$\Lm+\rho = (m+1,m-1,\ldots,2,1|1,2,\ldots,n-1,n+1).$$

For $m = n$, $s_{\Lm} = 2$ and $M_{\Lm} = \emptyset$.  We obtain $|\sdim(V)| = 2$.  For $m = n+1$, we have $|\Gamma_{\Lm}| = n-1$.  So $r < \df \g$, and our formula gives $0$.
Finally, for $m > n+1$, $s_{\Lm} = 1$ and 
\[  M_{\Lm}= \left\{ \epsilon_i-\epsilon_j \mid i,j \in \{1,2,\ldots,m-n-1\} \cup \{m-n+1\}  \right\}.   \]
Thus $\Lm+\rho-\rho_{\Lm}^{0} = \epsilon_1 -\epsilon_{m-n+1}$
which is the highest weight of the adjoint representation of $\g_{\Lm}$ and we get that $|\sdim(V)|=(m-n)^2 -1$ as required.

%
%

\end{document}